\documentclass[11pt]{amsart}
\usepackage[pdftex,width=145mm,height=220mm]{geometry}
\usepackage{amsfonts,amssymb,amscd,amsmath,enumerate,verbatim,calc}
\usepackage{amsthm}
\usepackage[colorlinks,citecolor=blue]{hyperref}

\newtheorem{theorem}{Theorem}[section]
\newtheorem{lemma}{Lemma}[section]
\newtheorem{pro}{Proposition}[section]
\newtheorem{corollary}{Corollary}[section]

\theoremstyle{remark}

\theoremstyle{remark}

\newtheorem{defi}{Definition}[section]

\providecommand{\norm}[1]{\lVert#1\rVert}

\DeclareMathOperator*{\tr}{trace}
\newcommand{\f}[2]{\frac{#1}{#2}}
\newcommand{\h}{\sigma}
\newcommand{\co}{\nabla}

\newcommand{\ff}{\rho}
\newcommand{\di}{\mathcal{D}}

\newcommand{\bco}{\widetilde{\nabla}}

\newcommand{\p}{\mathbf{P}}

\newcommand{\m}{{M_1 }_{\rho_2}\!\!\times_{\rho_1}M_2}
\let\~\tilde

\begin{document}
\title[On doubly warped product submanifolds of Generalized ...]{On doubly warped product submanifolds of Generalized $(\kappa,\mu)$-Space Forms }

\author{Morteza Faghfouri \and Narges Ghaffarzadeh }
\address{Faculty of mathematics, University of Tabriz, Tabriz, Iran}
\email{faghfouri@tabrizu.ac.ir}
\email{nargesghafarzade@yahoo.com}

\subjclass[2010]{53C40, 53C25, 53C15.}
\keywords{ doubly warped product, C-totally real submanifold, geometric inequality, eigenfunction of the Laplacian operator, $(\kappa,\mu)$-space forms.}
%
%
\begin{abstract}
In this paper we establish a general inequality involving the Laplacian of the warping functions and the squared mean curvature of any doubly warped product isometrically immersed in a Riemannian manifold. Moreover,  we obtain  some  geometric  inequalities for C-totally real doubly warped product submanifolds of generalized $(\kappa,\mu)$-space forms.
 \end{abstract}
\maketitle
\section{Introduction}
Bishop and O'Neill \cite{bishop.oneill:} introduced the concept of warped products to study manifolds of negative sectional  curvature.
 O'Niell  discussed warped products and explored curvature formulas of warped products in terms of curvatures of  components of warped products. Moreover, he studied Robertson-Walker, static, Schwarschild and Kruskal space-times as warped products in \cite{oneil:book}.
Doubly warped products can be considered as a generalization of singly warped products which were mainly studied in \cite{unal:doubly.warped.products.th,unal:doubly.warped.products}. A. Olteanu \cite{olteanu:generalinequalityfordoublywarped},  S.~Sular and C.~{\"O}zg{\"u}r \cite{sular.ozgur:Doublywarpedproduct}, K.~Matsumoto \cite{matsumoto:doublywarpedsubmanifold} and  M.~Faghfouri and A.~Majidi in \cite{faghfouri:ondoublyimmersion}  extended  some properties of warped product submanifolds  and geometric inequalities  in warped product manifolds  for doubly warped product submanifolds into Riemannian manifolds.

M. M. Tripathi \cite{tripathi:c-totallyrealwarped} studied
the relationship between the Laplacian of the
warping function $\ff$ and the squared mean curvature of a warped product $M=M_1\times_{\ff} M_2$ isometrically immersed
in a Riemannian manifold $\~M$ given by
\begin{align}\label{eq:chen-tripathineq0}
n_2\dfrac{\Delta_1\rho}{\rho}\leq\f{n^2}{4}\norm{H}^2+\tilde{\tau}(T_pM)-\tilde{\tau}(T_pM_1)-\tilde{\tau}(T_pM_2),
\end{align}
where $n_i = \dim M_i, i = 1, 2$, and $\Delta_1$ is the Laplacian operator of $M_1$.
Moreover, the equality case of \eqref{eq:chen-tripathineq0} holds  if and only if $\mathbf{x}$ is
a mixed totally geodesic immersion and $n_1H_1 = n_2H_2$, where $H_i, i = 1, 2$,
are the partial mean curvature vectors.
He also in Theorem 4.2 and Theorem 4.4 establish an inequality for C-totally real warped product
 submanifolds of $(\kappa,\mu)$-space forms and non-Sasakian
$(\kappa,\mu)$-manifolds.

 S.~Sular and C.~{\"O}zg{\"u}r \cite{sular.ozgur:Doublywarpedproduct} improved  Theorem 4.2 and Theorem 4.4 of M. M. Tripathi \cite{tripathi:c-totallyrealwarped} for C-totally real doubly warped product
 submanifolds of $(\kappa,\mu)$-space forms and non-Sasakian
$(\kappa,\mu)$-manifolds.

In \cite{carriazo.Tripathi:Generalizedspaceforms} A.~Carriazo, V.~Mart{\'{\i}}n Molina and M. M. Tripathi introduce generalized $(\kappa,\mu)$-space forms as an almost contact metric manifold $(\~M,\phi,\xi,\eta,g)$ whose
curvature tensor can be written as
\begin{align*}
R=f_1R_1+f_2R_2+f_3R_3+f_4R_4+f_5R_5+f_6R_6,
\end{align*}
where $f_1, f_2, f_3,f_4,f_5,f_6$  are differentiable functions on $\~M$, and $R_1, R_2, R_3, R_4, R_5, R_6$ are the tensors
defined by
\begin{align*}
&R_1(X,Y)Z=g(Y,Z)X-g(X,Z)Y,\\
&R_2(X,Y)Z=g(X,\phi Z)\phi Y-g(Y,\phi Z)\phi X+2g(X,\phi Y)\phi Z,\\
&R_3(X,Y)Z=\eta(X)\eta(Z)Y-\eta(Y)\eta(Z)X+g(X,Z)\eta(Y)\xi-g(Y,Z)\eta(X)\xi,\\
&R_4(X,Y)Z=g(Y,Z)hX-g(X,Z)hY+g(hY,Z)X-g(hX,Z)Y,\\
&R_5(X,Y)Z=g(hY,Z)hX-g(hX,Z)hY+g(\phi hX, Z)\phi hY-g(\phi hY,Z)\phi hX,\\
&R_6(X,Y)Z=\eta(X)\eta(Z)hY-\eta(Y)\eta(Z)hX+g(hX,Z)\eta(Y)\xi-g(hY,Z)\eta(X)\xi.
\end{align*}
for  vector fields $X, Y, Z$ on $\~M$. In \cite{carriazo:Generalizeddivided} A. Carriazo and V. Mart{\'{\i}}n-Molina defined generalized $(\kappa,\mu)$-space forms with divided the tensor field $R_5$ into two parts
 \begin{align*}
&R_{5,1}(X,Y)Z=g(hY,Z)hX-g(hX,Z)hY,\\
&R_{5,2}(X,Y)Z=g(\phi hY, Z)\phi hX-g(\phi hX,Z)\phi hY.
\end{align*}
It follows that $R_5 = R_{5,1}-R_{5,2}.$ They called  an almost contact metric manifold $(\~M,\phi,\xi,\eta,g)$, generalized $(\kappa,\mu)$-space forms with divided $R_5$ whenever the curvature tensor  can be written as
\begin{align*}
R=f_1R_1+f_2R_2+f_3R_3+f_4R_4+f_{5,1}R_{5,1}+f_{5,2}R_{5,2}+f_6R_6,
\end{align*}
where $f_1, f_2, f_3,f_4,f_{5,1}, f_{5,2},f_6$  are differentiable functions on $\~M$. Obviously, any  generalized Sasakian $(\kappa,\mu)$-space form is  a generalized Sasakian $(\kappa,\mu)$-space form  with divided $R_5.$

In section \ref{Inequality for doubly warped product}, we improved inequality \eqref{eq:chen-tripathineq0} for a doubly warped product isometrically
immersed in a Riemannian manifold. Section \ref{generalized} contains
some necessary background of contact geometry including the concepts of
Sasakian manifolds,  $(\kappa,\mu)$-space forms, generalized $(\kappa,\mu)$-space forms, generalized $(\kappa,\mu)$-space forms with divided $R_5$ and C-totally real submanifold. So we establish a similar inequality for C-totally real doubly warped product
submanifolds in a generalized $(\kappa,\mu)$-space forms with divided $R_5$.
\section{Preliminaries }
Let $M_1$ and $M_2$ be  two Riemannian manifolds equipped with Riemannian metrics $g_1$ and $g_2$, respectively, and let $\ff_1$ and $\ff_2$  be  positive differentiable functions on $M_1$ and $M_2$, respectively.  The doubly warped product ${M_1 }_{\ff_2}\!\!\times_{\ff_1}M_2$ is defined to be the product manifold $M_1\times M_2$ equipped with the Riemannian metric given by $$g=(\ff_2o\pi_2)^2\pi_1^*(g_1)+(\ff_1o\pi_1)^2\pi_2^*(g_2)$$ where $\pi_i:M_1\times M_2 \to M_i$ are the natural projections. We denote the dimension of $M_1$ and $M_2$ by $n_1$ and $n_2$, respectively. In particular, if $\ff_2=1$ then ${M_1 }_{1}\!\!\times_{\ff_1}M_2={M_1 }\times_{\ff_1}M_2$ is called warped product of $(M_1,g_1)$ and $(M_2,g_2)$ with warping function $\ff_1$.

For a vector field $X$ on $M_1$, the lift of $X$  to ${M_1 }_{\ff_2}\!\!\times_{\ff_1}M_2$ is the vector field $\tilde{X}$ whose value at each $(p,q)$ is the lift $X_p$ to $(p,q)$. Thus the lift of $X$ is the unique vector field on ${M_1 }_{\ff_2}\!\!\times_{\ff_1}M_2$ that is $\pi_1$-related to $X$ and $\pi_2$-related to the zero vector field on $M_2$.
For a doubly warped product ${M_1 }_{\ff_2}\!\!\times_{\ff_1}M_2$, let $\di_i$ denote the distribution obtained from the vectors tangent to the horizontal lifts of $M_i$.

Let $M$ be an n-dimensional Riemannian manifold equipped with a Riemannian
metric $g$.
Let $\{e_1,\ldots, e_n\}$ be any orthonormal basis for $T_pM$. The scalar curvature
$\tau(p)$ of $M$ at $p$ is defined by
\begin{align}
\tau(p)=\sum_{1\leq i<j\leq n} K(e_i\wedge e_j),
\end{align}
where $K(e_i\wedge e_j)$ is the sectional curvature of the plane section spanned by $e_i$ and $e_j$ at $p\in M$.

Let $\mathbf{P}_k$ be a $k$-plane section of $T_pM$ and $\{e_1,\ldots , e_k\}$
any orthonormal basis of $\mathbf{P}_k$. The scalar curvature $\tau(\mathbf{P}_k)$ of $\mathbf{P}_k$ is given by
\begin{align}
\tau(\mathbf{P}_k)=\sum_{1\leq i<j\leq k}K(e_i\wedge e_j).
\end{align}
The scalar curvature $\tau(p)$ of $M$ at $p$ is identical with the scalar curvature of
the tangent space $T_pM$ of $M$ at $p$, that is, $ \tau(p)=\tau(T_pM)$.

Let $\mathbf{x}:M\to \~M$ be an isometric immersion of a Riemannian manifold $M$ into a Riemannian manifold $\~M$. The formulas of Gauss and Weingarten are given respectively by
\begin{align}
\widetilde{\co}_XY&=\co_XY+\h(X,Y),\label{eq:gauss}\\
\bco_X\zeta&=-A_\zeta X+D_X\zeta,\label{eq:win}
\end{align}
for all vector fields $X,Y$ tangent to $M$ and $\zeta$ normal to $M$, where $\bco$ denotes the Levi-Civita connection on $\~M$, $\h$ the second fundamental form, $D$ the normal connection, and $A$ the shape operator of $\mathbf{x}:M\to \~M$. The second fundamental form and the shape operator are related by $\langle A_\zeta X,Y\rangle=\langle \h(X,Y),\zeta\rangle$, where $\langle,\rangle$ denotes the inner product on $\~M$.

The equation of Gauss of $\mathbf{x}:M\to \~M$ is given by
\begin{align}
\langle \widetilde{R}(X,Y)Z,W\rangle=\langle R(X,Y)Z,W\rangle+\langle \h(X,Z),\h(Y,W)\rangle\nonumber\\
\qquad-\langle \h(X,W),\h(Y,Z)\rangle,\label{eq:gauss curvature}
\end{align}
for $X, Y, Z, W\in\Gamma(TM)$.

If a Riemannian manifold $\~M$ is of constant curvature
$c$, we have
 \begin{align}
\langle R(X,Y)Z,W\rangle&=c\{\langle Y,Z \rangle\langle X,W\rangle-\langle X,Z \rangle\langle Y,W\rangle\}\label{eq:curvature:cons}\\
&\qquad +\langle \h(Y,Z), \h(X,W)\rangle-\langle \h(X,Z), \h(Y,W)\rangle.\nonumber
\end{align}

The mean curvature vector ${H}$ is defined by
${H}=\frac{1}{n}\text{trace }\h$.
An isometric immersion $\mathbf{x}:M\to \~M$ is called minimal immersion in $\~M$ if the mean curvature vector
vanishes identically.
Let $\psi$ be a smooth  function on a Riemannian n-manifold $M$. Then  the Hessian tensor field of $\psi$ is given by
\begin{align}\label{eq:hessian}
H^\psi(X,Y)=XY\psi-(\co_XY)\psi
\end{align}
and
the Laplacian of $\psi$ is given by
\begin{align}\label{eq:laplac}
\Delta\psi=-\text{trace }(H^\psi)=\sum_{i=1}^{n}((\co_{e_i}e_i)\psi-e_ie_i\psi).
\end{align}
We state the following  Lemmas for later uses.
\begin{lemma}[\cite{chen:2011pseudo}]\label{lemma:mini}
Every  minimal submanifold $M$ in a Euclidean space $\mathbb{R}^m$
is non-compact.
\end{lemma}
\begin{lemma}[\cite{chen:2011pseudo}]\label{lemma:hopf1}
Every harmonic function on a compact Riemannian manifold
is  constant.
\end{lemma}
\begin{lemma}[Hopf's lemma. in \cite{chen:2011pseudo} ]\label{lemma:hopf2}
Let $M$ be a compact Riemannian manifold. If $\psi$ is a
differentiable function on $M$ such that $\Delta\psi \geq0$  everywhere on $M$ (or $\Delta\psi \leq0$
everywhere on $M$), then $\psi$ is a constant function.
\end{lemma}
\begin{lemma}[B. Y. Chen \cite{chen:Somepinching}]
Let $l\geq2$ and $a_1,\ldots, a_l, b$ be real numbers such that
\begin{align}
\left(\sum_{i=1}^la_i\right)^2=(l-1)\left(\sum_{i=1}^la_i^2+b\right).
\end{align}
Then $2a_1a_2\geq b$, with equality holding if and only if $a_1+a_2=a_3=\cdots=a_l.$
\end{lemma}

\section{Inequality for doubly warped products}\label{Inequality for doubly warped product}
\begin{pro}
Let $\mathbf{x}$ be an isometric immersion of an n-dimensional doubly warped product manifold $M=\m$ into an m-dimensional Riemannian manifold $\tilde{M}$. Then
\begin{align}\label{eq:chen-tripathineq}
n_2\dfrac{\Delta_1\rho_1}{\rho_1}+n_1\dfrac{\Delta_2\ff_2}{\ff_2}\leq\f{n^2}{4}\norm{H}^2+\tilde{\tau}(T_pM)-\tilde{\tau}(T_pM_1)-\tilde{\tau}(T_pM_2)
\end{align}
where $n_i = \dim M_i, i = 1, 2$, and $\Delta_i$ is the Laplacian operator of $M_i$.
Moreover, the equality case of \eqref{eq:chen-tripathineq} holds identically if and only if $\mathbf{x}$ is
a mixed totally geodesic immersion and $n_1H_1 = n_2H_2$, where $H_i, i = 1, 2$,
are the partial mean curvature vectors.
\end{pro}
\begin{proof}
We choose a local orthonormal frame $\{e_1,\ldots,e_{n_1},e_{n_1+1},\ldots,e_m\}$, such that $e_1,\ldots,e_{n_1}$ are tangent to $M_1$, $e_{n_1+1},\ldots,e_n$ are tangent to $M_2$ and $e_{n+1}$ is parallel to mean curvature vector $H$.
We put $\sigma_{ij}^r=\langle\sigma(e_i,e_j),e_r\rangle, \quad i,j\in\{1, \ldots,n\}, r\in\{n+1,...,m\}$ where $n=n_1+n_2$ and \begin{align}
\norm{\sigma}^2&=\sum_{i,j=1}^n\langle\sigma(e_i,e_j),\sigma(e_i,e_j)\rangle\nonumber\\
&=\sum_{i,j=1}^n\sum_{r=n+1}^m(\sigma_{ij}^r)^2.
\end{align}
In view of the equation \eqref{eq:gauss curvature}, we have
\begin{align}\label{eq:10}
K(e_{i}\wedge e_{j})=\~{K}(e_{i}\wedge e_{j})+\sum_{r=n+1}^{m}(\sigma_{ii}^{r}\sigma_{jj}^{r}-(\sigma_{ij}^{r})^{2})
\end{align}
From \eqref{eq:10}, we have
\begin{align}\label{eq:11}
2\tau(p)=n^2\norm{H}^2-\norm{\sigma}^2+2\~\tau(T_pM).
\end{align}
We set
\begin{align}\label{2.8}
2\delta=4\tau(p)-4\~\tau(T_pM)-n^2\norm{H}^2.
\end{align}
The equation \eqref{2.8} can be written as
\begin{align}
n^2\norm{H}^2=2(\delta+\norm{\sigma}^2).\label{eq:12}
\end{align}
For the chosen locally orthonormal frame, we have
\begin{align*}
n^2\norm{H}^2&=\langle n H,n H\rangle=\langle n\norm{H}e_{n+1},\sum_{i=1}^n\sum_{r=n+1}^m\sigma_{ii}^re_r\rangle\\
&=n\norm{H}\sum_{i=1}^n\sigma_{ii}^{n+1},
\end{align*}
$$n\norm{H}=\sum_{i=1}^n\sigma_{ii}^{n+1}$$
\begin{align*}
(\sum_{i=1}^n\sigma_{ii}^{n+1})^2&=2(\delta+\sum_{i=1}^n(\sigma_{ii}^{n+1})^2+\sum_{i\neq j}^n(\sigma_{ij}^{n+1})^2+\sum_{i,j=1}^n\sum_{r=n+2}^m(\sigma_{ij}^r)^2)\\
\end{align*}
If we put $a_1=\sigma_{11}^{n+1}, a_2=\sum_{i=2}^{n_1}\sigma_{ii}^{n+1}$ and $a_3=\sum_{t=n_1+1}^{n}\sigma_{tt}^{n+1}$, then
\begin{align}
(\sum_{i=1}^3a_i)^2=2(\sum_{i=1}^3a_i^2+b)\label{eq:chen}
\end{align}
where
\begin{align*}
b=\delta+&\sum_{1\leq i\neq j\leq n}(\sigma_{ij}^{n+1})^2-\sum_{2\leq j\neq k\leq n_1}\sigma_{jj}^{n+1}\sigma_{kk}^{n+1}\\
-&\sum_{n_1+1\leq s\neq t\leq n}\sigma_{tt}^{n+1}\sigma_{ss}^{n+1}+\sum_{i,j=1}^n\sum_{r=n+2}^m(\sigma_{ij}^r)^2
\end{align*}
Applying Chen's Lemma for $n=3$
we get $b\leq2a_1a_2$, with equality holding if and only if $a_1+a_2=a_3$. Equivalently, we get
\begin{align}\label{2.12}
\frac{\delta}{2}+\sum_{1\leq i<j\leq n}(\sigma_{ij}^{n+1})^2+\frac{1}{2}\sum_{i,j=1}^n\sum_{r=n+2}^m(\sigma_{ij}^r)^2\leq\sum_{1\leq j<k\leq n_1}\sigma_{jj}^{n+1}\sigma_{kk}^{n+1}+\sum_{n_1+1\leq s< t\leq n}\sigma_{ss}^{n+1}\sigma_{tt}^{n+1}
\end{align}
with equality holding if and only if
\begin{align}\label{2.13}
\sum_{i=1}^{n_1}\sigma_{ii}^{n+1}=\sum_{s=n_1+1}^{n}\sigma_{ss}^{n+1}
\end{align}
In \cite{olteanu:generalinequalityfordoublywarped}, Olteanu  for doubly warped product manifolds proved that  for unit vector fields $X\in\di_1$ and $Z\in\di_2$ we have:
\begin{align} \label{eq:k}
K(X\wedge Z)=\frac{1}{\ff_1}((\co^1_XX)\ff_1-X^2\ff_1)+\frac{1}{\ff_2}((\co^2_ZZ)\ff_2-Z^2\ff_2).
\end{align}
For the chosen locally orthonormal frame and \eqref{eq:k}, we have
\begin{align}\label{eq:19}
n_2\dfrac{\Delta_1\rho_1}{\rho_1}+n_1\dfrac{\Delta_2\ff_2}{\ff_2}=\sum_{1\leq i\leq  n_{1}<j\leq n}
{K}(e_{i}\wedge e_{j}).
\end{align}
From equation \eqref{eq:19} we get
\begin{align}
n_2\dfrac{\Delta_1\rho_1}{\rho_1}+n_1\dfrac{\Delta_2\ff_2}{\ff_2}=\tau(p)-\tau(T_pM_1)-\tau(T_pM_2).
\end{align}
Using the Gauss equation \eqref{eq:gauss curvature}, we have
\begin{align}
n_2\dfrac{\Delta_1\rho_1}{\rho_1}+n_1\dfrac{\Delta_2\ff_2}{\ff_2}=&\tau(p)-\~\tau(T_pM_1)
-\sum_{r=n+1}^m\sum_{1\leq j<i\leq n_1}(\sigma_{jj}^r\sigma_{ii}^r-(\sigma_{ji}^r)^2)\label{2.14}\\
&-\~\tau(T_pM_2)-\sum_{r=n+1}^m\sum_{n_1+1\leq s<t\leq n}(\sigma_{ss}^r\sigma_{tt}^r-(\sigma_{st}^r)^2).\nonumber
\end{align}
From \eqref{2.8} we get
\begin{align}
n_2\dfrac{\Delta_1\rho_1}{\rho_1}+n_1\dfrac{\Delta_2\ff_2}{\ff_2}=&\dfrac{n^{2}}{4}\|H\|^{2}+\frac{\delta}{2}-\~\tau(T_pM_1)
-\sum_{r=n+1}^m\sum_{1\leq j<i\leq n_1}(\sigma_{jj}^r\sigma_{ii}^r-(\sigma_{ji}^r)^2)\label{eq:bala}\\
&-\~\tau(T_pM_2)-\sum_{r=n+1}^m\sum_{n_1+1\leq s<t\leq n}(\sigma_{ss}^r\sigma_{tt}^r-(\sigma_{st}^r)^2).\nonumber
\end{align}
In view of \eqref{2.12} and \eqref{eq:bala} we get
\begin{align}
n_2\dfrac{\Delta_1\rho_1}{\rho_1}+n_1\dfrac{\Delta_2\ff_2}{\ff_2}\leq&\dfrac{n^{2}}{4}\|H\|^{2}+\tilde{\tau}(T_{p}M)-\tilde{\tau}(T_{p}M_{1})-\tilde{\tau}(T_{p}M_{2})\nonumber\\
&-\sum_{r=n+2}^m\sum_{n_1+1\leq s<t\leq n}(\sigma_{ss}^r\sigma_{tt}^r-(\sigma_{st}^r)^2)\\
&-\sum_{r=n+2}^m\sum_{1\leq j<i\leq n_1}(\sigma_{jj}^r\sigma_{ii}^r-(\sigma_{ji}^r)^2)\nonumber\\
&-\sum_{j=1}^{n_1}\sum_{t=n_1+1}^n(\sigma_{jt}^{n+1})^2-\frac{1}{2}\sum_{r=n+2}^{m}\sum_{i,j=1}^n(\sigma_{ij}^{r})^2,\nonumber
\end{align}
or
\begin{align}\label{2.15}
n_2\dfrac{\Delta_1\rho_1}{\rho_1}+n_1\dfrac{\Delta_2\ff_2}{\ff_2}\leq&\dfrac{n^{2}}{4}\|H\|^{2}+\tilde{\tau}(T_{p}M)-\tilde{\tau}(T_{p}M_{1})-\tilde{\tau}(T_{p}M_{2})\\
&-\sum_{r=n+1}^m\sum_{j=1}^{n_1}\sum_{t=n_1+1}^n(\sigma_{tj}^r)^2-\frac{1}{2}\sum_{r=n+2}^m\left(\sum_{j=1}^{n_1}\sigma_{jj}^r\right)^2\nonumber\\
&-\frac{1}{2}\sum_{r=n+2}^m\left(\sum_{t=n_1+1}^{n}\sigma_{tt}^r\right)^2,\nonumber
\end{align}
which implies the inequality \eqref{eq:chen-tripathineq}.

The equality holds in  \eqref{2.15} if and only if
\begin{align}\label{2.16}
\sigma_{jt}^r=0, \quad 1\leq j\leq n_1, n_1+1\leq t\leq n, n+1\leq r\leq m,
\end{align}
and
\begin{align}\label{2.17}
\sum_{i=1}^{n_1}\sigma_{ii}^r=0=\sum_{t=n_1+1}^{n}\sigma_{tt}^r, \quad n+2\leq r\leq m.
\end{align}
Obviously \eqref{2.16} is true if and only if the doubly warped product $\m$ is mixed totally geodesic. From the equations \eqref{2.13} and \eqref{2.17} it follows that $n_1H_1=n_2H_2$.

The converse statement is straightforward.
\end{proof}
\section{generalized $(\kappa,\mu)$-space forms} \label{generalized}
A $(2m + 1)$-dimensional differentiable manifold $\~M$ is called an almost contact
metric manifold if there is an almost contact metric structure $(\phi,\xi,\eta,g)$
consisting of a $(1, 1)$ tensor field $\phi$, a vector field $\xi$, a 1-form $\eta$ and a compatible
Riemannian metric $g$ satisfying
\begin{gather}
\phi^2=-I+\eta\otimes\xi, \eta(\xi)=1, \phi\xi=0, \eta o\phi=0,\\
g(X,Y)=g(\phi x,\phi Y)+\eta(X)\eta(Y),\\
g(X,\phi y)=-g(\phi X,Y), g(X,\xi)=\eta(X),
\end{gather}
for all $X,Y\in\Gamma(T\~M)$.
An almost contact metric structure becomes a contact metric structure if $d\eta=\Phi$, where $\Phi(X,Y)=g(X,\phi Y)$ is the fundamental 2-form of $\~M$.

An almost contact metric structure of $\~M$ is said
to be normal if the Nijenhuis torsion $[\phi,\phi]$ of $\phi$ equals $-2d\eta\otimes\xi$. A normal
contact metric manifold is called a Sasakian manifold. It can be proved that
an almost contact metric manifold is Sasakian if and only if
\begin{align}
(\co_X\phi)Y=g(X,Y)\xi-\eta(Y)X,
\end{align}
for any $X, Y\in \Gamma(T\~M)$ or  equivalently, a contact metric structure is a Sasakian structure if and only if $\~R$ satisfies
\begin{align}
\~R(X,Y)\xi=\eta(Y)X-\eta(X)Y,
\end{align}
for $X,Y\in\Gamma(T\~M)$.
In a contact metric manifold $\~M$ , the $(1, 1)$-tensor field $h$ is  defined by $2h =\mathcal{L}_\xi\phi$, which is the Lie derivative of $\phi$ in the characteristic direction $\phi$. It is
symmetric and satisfies
\begin{gather}
h\xi=0,\quad h\phi+\phi h=0,\\
\~\co\xi=-\phi-\phi h,\quad  \tr(h)=\tr(\phi h)=0,
\end{gather}
where $\~\co$ is Levi-Civita connection.

Given an almost contact metric manifold $(\phi,\xi,\eta,g)$, a $\phi$-section of $M$ at $p\in M$ is a section $\p\subset T_p\~M$ spanned by a unit vector $X_p$ orthogonal to $\xi_p$, and $\phi X_p$. The $\phi$-sectional curvature of $\p$ is defined by $\~K(X,\phi X)=\~R(X,\phi X,\phi X,X)$. A Sasakian manifold with constant $\phi$-sectional curvature $c$
is called a Sasakian space form and is denoted by
$\~M(c).$
A contact metric manifold $(\~M,\phi,\xi,\eta,g)$ is said to be a $(\kappa,\mu)$-contact manifold if its curvature tensor satisfies the condition
\begin{align}
\~R(X,Y)\xi=\kappa(\eta(Y)X-\eta(X)Y)+\mu(\eta(Y)hX-\eta(X)hY),
\end{align}
where  $\kappa$ and $\mu$ are real constant numbers.
If the $(\kappa,\mu)$-contact metric manifold $\~M$ has constant $\phi$-sectional curvature $c$, then it is said  to be a $(\kappa,\mu)$-contact space form.
\begin{defi}[\cite{carriazo.Tripathi:Generalizedspaceforms}]
We say that an almost contact metric manifold $(\~M,\phi,\xi,\eta,g)$ is a generalized $(\kappa,\mu)$-space form if there exist functions $f_1, f_2, f_3,f_4,f_5,f_6$ defined on $M$ such that
\begin{align}\label{eq:km-space}
R=f_1R_1+f_2R_2+f_3R_3+f_4R_4+f_5R_5+f_6R_6,
\end{align}
where $R_1, R_2, R_3, R_4, R_5, R_6$ are the following tensors
\begin{align*}
&R_1(X,Y)Z=g(Y,Z)X-g(X,Z)Y,\\
&R_2(X,Y)Z=g(X,\phi Z)\phi Y-g(Y,\phi Z)\phi X+2g(X,\phi Y)\phi Z,\\
&R_3(X,Y)Z=\eta(X)\eta(Z)Y-\eta(Y)\eta(Z)X+g(X,Z)\eta(Y)\xi-g(Y,Z)\eta(X)\xi,\\
&R_4(X,Y)Z=g(Y,Z)hX-g(X,Z)hY+g(hY,Z)X-g(hX,Z)Y,\\
&R_5(X,Y)Z=g(hY,Z)hX-g(hX,Z)hY+g(\phi hX, Z)\phi hY-g(\phi hY,Z)\phi hX,\\
&R_6(X,Y)Z=\eta(X)\eta(Z)hY-\eta(Y)\eta(Z)hX+g(hX,Z)\eta(Y)\xi-g(hY,Z)\eta(X)\xi.
\end{align*}
for all vector fields $X, Y, Z $ on $\~M$, where $2h=L_\xi \phi$ and $L$  is the usual Lie derivative. We will denote such a manifold by $\~M(f_1,\ldots,f_6)$.
\end{defi}
$(\kappa,\mu)-$space forms are examples of generalized $(\kappa,\mu)$-space forms, with constant functions
$$f_1=\frac{c+3}{4}, f_2=\frac{c-1}{4}, f_3=\frac{c+3}{4}-\kappa, f_4=1, f_5=\frac{1}{2}, f_6=1-\mu. $$
Generalized Sasakian space forms $\~M(f_1,f_2,f_3)$ introduced in \cite{alegre.blair:GeneralizedSasakianspaceforms} are generalized $(\kappa,\mu)$-space forms, with $f_4=f_5=f_6=0.$

\begin{defi}[\cite{carriazo:Generalizeddivided}]
We say that an almost contact metric manifold $(\~M,\phi,\xi,\eta,g)$ is a generalized $(\kappa,\mu)$-space form with divided $R_5$ if there exist function $f_1, f_2, f_3,$ $ f_4, f_{5,1}, f_{5,2}, f_6$ defined on $M$ such that
\begin{align}\label{eq:km-spacedvid}
R=f_1R_1+f_2R_2+f_3R_3+f_4R_4+f_{5,1}R_{5,1}+f_{5,2}R_{5,2}+f_6R_6,
\end{align}
where $R_{5,1}, R_{5,2}$ are the following tensors
\begin{align*}
&R_{5,1}(X,Y)Z=g(hY,Z)hX-g(hX,Z)hY,\\
&R_{5,2}(X,Y)Z=g(\phi hY, Z)\phi hX-g(\phi hX,Z)\phi hY,
\end{align*}
for all vector fields $X, Y, Z $ on $\~M$, where $2h=L_\xi \phi$ and $L$  is the usual Lie derivative. we will denote such a manifold by $\~M(f_1,f_2,f_3,f_4,f_{5,1},f_{5,2},f_6)$.
\end{defi}
It follows that $R_5 = R_{5,1}-R_{5,2}$.  It is obvious that, if $\~M(f_1,\ldots,f_6)$  is  a generalized $(\kappa,\mu)$-space form then $\~M$ is a generalized $(\kappa,\mu)$-space form with divided $R_5$ with $f_{5,1}=f_5$ and $f_{5,2}=-f_5$.
A non-Sasakian $(\kappa,\mu)$-space form is the generalized $(\kappa,\mu)$-space form with divided $R_5$ with $$f_1=\frac{2-\mu}{2}, f_2=-\frac{\mu}{2},f_3=\frac{2-\mu-2\kappa}{2},f_4=1, f_{5,1}=\frac{2-\mu}{2(1-\kappa)}, f_{5,2}=\frac{2\kappa-\mu}{2(1-\kappa)}$$ and $f_6=1-\mu$ but not the generalized $(\kappa,\mu)$-space form.

A submanifold $M$ in a contact manifold is called a C-totally real submanifold if every tangent vector of $M$ belongs to the contact distribution \cite{yamaguchi:c-totallyrealsubmanifolds}.
Thus, a submanifold $M$ in a contact metric manifold is a C-totally real submanifold if $\xi$ is normal to $M$. A submanifold $M$ in an almost contact metric
manifold is called anti-invariant  if $\phi(TM)\subset T^\bot(M)$ \cite{yano:Antiinvariantsubmanifolds}. If a submanifold
$M$ in a contact metric manifold is normal to the structure vector field $\xi$,
then it is anti-invariant. Thus C-totally real submanifolds in a contact
metric manifold are anti-invariant, as they are normal to $\xi$

For a C-totally real submanifold in a contact metric manifold we have
$$\langle A_\xi X,Y\rangle=-\langle\~\co_X\xi,Y\rangle=\langle\phi X+\phi hX,Y\rangle,$$ which implies that
\begin{align}\label{4.1}
A_\xi=(\phi h)^T
\end{align}
where $(\phi h)^T$ is the tangential part of $\phi h X$ for all $X\in\Gamma(TM)$.
Now, we obtain a basic inequality involving the Laplacian of the warping
function and the squared mean curvature of a C-totally real warped product
submanifold of a $(\kappa,\mu)$-space form.

\begin{theorem}
Let $M=\m$ be an $n$-dimensional $C$-totally real doubly warped product submanifold of a $(2m+1)$-dimensional generalized $(\kappa,\mu)$-space form with divided $R_5$  $\~M(f_1,\ldots,f_6)$. Then
\begin{align}\label{eq:dwin}
n_2\dfrac{\Delta_1\rho_1}{\rho_1}+n_1\dfrac{\Delta_2\ff_2}{\ff_2}\leq\f{n^2}{4}&\norm{H}^2 +n_1n_2f_1\nonumber\\
&+f_4\bigg(n_2\tr(h^T_{|M_1})+n_1\tr(h^T_{|M_2})\bigg)\nonumber\\
&+\f{1}{2}f_{5,1}\bigg((\tr(h^T))^2-(\tr(h^T_{|M_1}))^2-(\tr(h^T_{|M_2}))^2\\
&\qquad-\norm{h^T}^2+\norm{h^T_{|M_1}}^2+\norm{h^T_{|M_2}}^2\bigg)\nonumber\\
&-\f{1}{2}f_{5,2}\bigg((\tr(A_\xi))^2-(\tr(A_{\xi|M_1}))^2-(\tr(A_{\xi|M_2}))^2\nonumber\\
&\qquad-\norm{A_\xi}^2+\norm{A_{\xi|M_1}}^2+\norm{A_{\xi|M_2}}^2\bigg)\nonumber
\end{align}
where $n_i=\dim M_i, n=n_1+n_2 $ and $\Delta_i$ is the Laplacian of $M_i, i=1,2.$ Equality holds in \eqref{eq:dwin} identically if and only if $M$ is mixed totally geodesic and $n_1H_1=n_2H_2,$ where $H_i, i=1,2$, are the partial mean curvature vectors.
\end{theorem}
\begin{proof}
We choose a local orthonormal frame $\{e_1,\ldots,e_{n},e_{n+1},\ldots,e_{2m+1}\}$ such that $e_1,\ldots,e_{n_1}$ are tangent to $M_1$, $e_{n_1+1},\ldots,e_n$ are tangent to $M_2$ and $e_{n+1}$ is parallel to the mean curvature vector $H.$
Then from \eqref{eq:km-space} and
\eqref{4.1} we have
\begin{align}\label{4.3}
\~K(e_i\wedge e_j)=f_1&+f_4(g(h^Te_i,e_i)+g(h^Te_j,e_j))\nonumber\\
&+f_{5,1}(g(h^Te_i,e_i)g(h^Te_j,e_j)-g(h^Te_i,e_j)^2)\\
&-f_{5,2}(g(A_\xi e_i,e_i)g(A_\xi e_j,e_j)-g(A_\xi e_i,e_j)^2),\nonumber
\end{align}
where $h^TX$ is the tangential part of $hX$ for $X\in\Gamma(TM)$.
For a $k$-plane section $\p$ spanned by $\{e_1,\ldots,e_k\}$, from \eqref{4.3} it follows that
\begin{align}\label{4.4}
\~\tau(\p)=&\frac{k(k-1)}{2}f_1+(k-1)f_4\tr(h^T_{|\p})\\
&+\f{f_{5,1}}{2}\{\tr(h^T_{|\p})^2-\norm{h^T_{|\p}}^2)\}-\f{f_{5,2}}{2}\{\tr({A_\xi}_{|\p})^2-\norm{{A_\xi}_{|\p}}^2\}.\nonumber
\end{align}
From \eqref{4.4} we have
\begin{align}\label{4.5}
\~\tau(T_pM)=&\f{n(n-1)}{2}f_1+(n-1)f_4\tr(h^T)\\
&+\f{f_{5,1}}{2}\{\tr(h^T)^2-\norm{h^T}^2\}-\f{f_{5,2}}{2}\{\tr({A_\xi})^2-\norm{{A_\xi}}^2\},\nonumber
\end{align}
and
\begin{align}\label{4.6}
\~\tau(T_pM_i)=&\f{n_i(n_i-1)}{2}f_1+(n_i-1)f_4\tr(h^T_{|M_i})\nonumber\\
&+\f{f_{5,1}}{2}\{\tr(h^T_{|M_i})^2-\norm{h^T_{|M_i}}^2\}-\f{f_{5,2}}{2}\{\tr({A_\xi}_{|M_i})^2-\norm{{A_\xi}_{|M_i}}^2\},
\end{align}where
$i=1,2.$ By using \eqref{4.5} and \eqref{4.6} in \eqref{eq:chen-tripathineq} we get \eqref{eq:dwin}.
\end{proof}
\begin{corollary}\label{pro:4.3}
Let
$ \m $
be a doubly warped product of two Riemannian manifolds whose warping functions
$ \ff_{1} $ and $ \ff_{2} $
are harmonic functions. Then
\begin{enumerate}
\item
$ \m $
admits no  minimal C-totally real immersion into a $(2m+1)$-dimensional generalized $(\kappa,\mu)$-space form with divided $R_5$ $\~M(f_1,\ldots,f_6)$ with
\begin{align}\label{4.13}
0>& n_1n_2f_1 +f_4\bigg(n_2\tr(h^T_{|M_1})+n_1\tr(h^T_{|M_2})\bigg)\nonumber\\
&+\f{f_{5,1}}{2}\bigg((\tr(h^T))^2-(\tr(h^T_{|M_1}))^2-(\tr(h^T_{|M_2}))^2\nonumber\\
&\qquad-\norm{h^T}^2+\norm{h^T_{|M_1}}^2+\norm{h^T_{|M_2}}^2\\
&-\f{f_{5,2}}{2}\bigg(\tr(A_\xi))^2-(\tr(A_{\xi|M_1}))^2-(\tr(A_{\xi|M_2}))^2\nonumber\\
&\qquad-\norm{A_\xi}^2+\norm{A_{\xi|M_1}}^2+\norm{A_{\xi|M_2}}^2\bigg).\nonumber
\end{align}

\item
every minimal C-totally real immersion  of $\m$ into a $(2m+1)$-dimensional generalized $(\kappa,\mu)$-space form with divided $R_5$  $\~M(f_1,\ldots,f_6)$ with
\begin{align}\label{4.14}
0=& n_1n_2f_1 +f_4\bigg(n_2\tr(h^T_{|M_1})+n_1\tr(h^T_{|M_2})\bigg)\nonumber\\
&+\f{f_{5,1}}{2}\bigg((\tr(h^T))^2-(\tr(h^T_{|M_1}))^2-(\tr(h^T_{|M_2}))^2\nonumber\\
&\qquad-\norm{h^T}^2+\norm{h^T_{|M_1}}^2+\norm{h^T_{|M_2}}^2\\
&-\f{f_{5,2}}{2}\bigg(\tr(A_\xi))^2-(\tr(A_{\xi|M_1}))^2-(\tr(A_{\xi|M_2}))^2\nonumber\\
&\qquad-\norm{A_\xi}^2+\norm{A_{\xi|M_1}}^2+\norm{A_{\xi|M_2}}^2\bigg)\nonumber
\end{align}
  is a mixed totally geodesic immersion.
\end{enumerate}
\end{corollary}
\begin{proof}
Assume that
$ \phi:\m\to \tilde{M}$
is a C-totally real minimal immersion of a doubly warped product
$\m$
into a $(2m+1)$-dimensional generalized $(\kappa,\mu)$-space form with divided $R_5$  $\~M(f_1,\ldots,f_6)$.
 If
$ \ff_{1} $ and $ \ff_{2} $
are harmonic functions on
$ M_{1} $ and $ M_{2} $, respectively, then inequality \eqref{eq:dwin}
implies
\begin{align}
0\leq& n_1n_2f_1 +f_4\bigg(n_2\tr(h^T_{|M_1})+n_1\tr(h^T_{|M_2})\bigg)\nonumber\\
&+\f{f_{5,1}}{2}\bigg((\tr(h^T))^2-(\tr(h^T_{|M_1}))^2-(\tr(h^T_{|M_2}))^2\nonumber\\
&\qquad-\norm{h^T}^2+\norm{h^T_{|M_1}}^2+\norm{h^T_{|M_2}}^2\nonumber\\
&-\f{f_{5,2}}{2}\bigg((\tr(A_\xi))^2-(\tr(A_{\xi|M_1}))^2-(\tr(A_{\xi|M_2}))^2\nonumber\\
&\qquad-\norm{A_\xi}^2+\norm{A_{\xi|M_1}}^2+\norm{A_{\xi|M_2}}^2\bigg)\nonumber
\end{align}
on
the doubly warped product
$ \m $. This shows that $\m$
does not admit any C-totally real minimal immersion into $\~M$ with condition \eqref{4.13}.

To prove (2),
when
\eqref{4.14} is true, the minimality of $\m$ and the harmonicity  of $\ff_1$ and $\ff_2$ imply that
the equality in (\ref{eq:dwin}) holds identically. Thus, the immersion is mixed totally geodesic.
\end{proof}
\begin{corollary}
Let
$ \m $
be a doubly warped product of two Riemannian manifolds whose warping functions
$ \ff_{1} $ and $ \ff_{2} $
are harmonic functions and one of $M_i, i=1,2$ is compact. Then
every C-totally real minimal immersion from
$ \m$
into the generalized Sasakian space form $\mathbb{R}^{2m+1}(-3)$ is a warped product immersion.
\end{corollary}

\begin{proof}
Let $\phi:\m\to \mathbb{R}^{2m+1}(-3)$ be a C-totally real minimal immersion and $M_2$ be compact. Since $\ff_2$ is harmonic, by applying lemma \ref{lemma:hopf1} and the compactness of $M_2$, we know that $\ff_2$ is a positive constant. Therefore, the doubly warped product $\m$ can be expressed as a warped product $\bar{M}_1\times_{\ff_1} M_2$ where $\bar{M}_1=M_1$, equipped with the metric $\ff_2^2g_1$ which is homothetic to the original metric $g_1$ on $M_1$. Now,  Theorem 5.2  in  \cite{chen:on.isometric.minimal.immersions} implies that $\phi$ is a warped product immersion.
\end{proof}

\begin{pro}
If
$ \ff_{1} $ and $ \ff_{2} $ are eigenfunctions of the Laplacian on
$ M_{1} $ and $ M_{2} $ with eigenvalues
$ n_{1}\lambda $ and $ n_{2}\lambda, \lambda>0,$ (or with eigenvalues $\lambda$)   respectively. Then
$ \m $
does not admit a C-totally real minimal immersion
into a generalized $(\kappa,\mu)$-space form with divided $R_5$ $\~M(f_1,f_2,f_3,0,0,0,f_6)$ with non-positive $f_1$.
\end{pro}
\begin{proof}
The inequality \eqref{eq:dwin} implies that
$ n_1n_2f_1\geq \lambda>0. $
which in  contraction to $f_1\leq0$.
\end{proof}

\begin{pro}
If $M_1$  is a compact Riemannian manifold and $\ff_2$ is a harmonic function on $M_2$, then
every doubly warped product $\m$  does not admit a C-totally real minimal immersion into the standard  generalized Sasakian space form $\~M=\mathbb{R}^{2m+1}(-3)$.
\end{pro}
\begin{proof}
Assume $M_1$ is compact and $\phi:\m\to \~M$ is a C-totally real minimal immersion of $\m$ into the standard  generalized Sasakian space form $\~M$ with  $f_1=0$. From harmonicity  of $\ff_2$ and inequality \eqref{eq:dwin} we have
$$\frac{\Delta_1 \ff_1}{\ff_1}\leq n_1n_2f_1=0.$$
Since the warping function $\ff_1$ is positive, we obtain $\Delta_1 \ff_1\leq0$. Hence. it follows from Hopf's lemma \ref{lemma:hopf2} that $\ff_1$ is a positive constant.

Hence the equality case of \eqref{eq:dwin} holds and  $\phi$ is mixed totally geodesic.
Since $\ff_1$ is a positive constant, then  doubly warped product $\m$  is a warped product of the Riemannian manifold $(M_1,g_1)$ and the Riemannian manifold $\bar{M}_2=(M_2,\ff_1^2g_2)$, that is ${M_{1}}_{\ff_2}\!\!\times\bar{M}_2$.
By applying a result of N\"{o}lker in \cite{nolker:Isometric.immersions.of.warped.products}, $\phi$  is warped product immersion, say
$$\phi=(\phi_1,\phi_2):{M_{1}}_{\ff_2}\!\!\times\bar{M}_2\to \~M=\mathbb{R}^{n_1}\times \mathbb{R}^{n_2}$$
By Theorem 3.5 in \cite{faghfouri:ondoublyimmersion}, $\phi$ is minimal, $\phi_1:M_1\to\mathbb{R}^{n_1}$ is  minimal since $\phi$ is minimal. This is impossible by Lemma \ref{lemma:mini} since $M_1$ is compact.

\end{proof}

\bibliographystyle{acm}
\bibliography{F:/bibtex/convolution,F:/bibtex/faghfouri}

\end{document}